\documentclass[12pt, a4paper,DIV=11]{scrartcl}

\usepackage{amsthm, amsfonts, amsmath, latexsym, mathrsfs}
\usepackage{mathtools}
\usepackage{amssymb}
\usepackage[utf8]{inputenc}
\usepackage[T1]{fontenc}

\usepackage[sb]{libertine}

\usepackage{textcomp}
\usepackage{authblk}
\usepackage[pdftex, bookmarks=true]{hyperref}
\usepackage{enumitem}

\theoremstyle{definition}
\newtheorem{definition}{Definition}

\theoremstyle{remark}
\newtheorem{remark}[definition]{Remark}

\theoremstyle{plain}
\newtheorem{theorem}[definition]{Theorem}
\newtheorem{result}[definition]{Result}
\newtheorem{lemma}[definition]{Lemma}

\newtheorem{corollary}[definition]{Corollary}

\newcommand{\eps}{\varepsilon}

\newcommand{\disk}{\mathbb{D}}

\newcommand{\hol}{\mathcal{O}}

\newcommand{\bcdot}{\boldsymbol{\cdot}}

\newcommand{\C}{\mathbb{C}}

\newcommand{\D}{\mathbb{D}}

\makeatletter
  \def\blfootnote{\gdef\@thefnmark{}\@footnotetext}
\makeatother

\begin{document}
	\title{Conformal Mappings Through the Lens of Invariant Metrics}
	\author{ Bharathi Thiruvengadam \thanks{Bharathi is supported by a fellowship
from the University Grants Commission, India.\\
  \href{mailto:212114004@smail.iitpkd.ac.in}{212114004@smail.iitpkd.ac.in}\\
Department of Mathematics, Indian Institute of Technology Palakkad,
Palakkad, Kerala- 678623, India} \and Jaikrishnan Janardhanan \thanks{
  \href{mailto:jaikrishnan@iitpkd.ac.in}{jaikrishnan@iitpkd.ac.in}\\
Department of Mathematics, Indian Institute of Technology Palakkad,
Palakkad, Kerala- 678623, India}}

	\maketitle

	\begin{abstract}
    The main objective of this paper is to show that balls under invariant
    metrics on hyperbolic planar domains are finitely-connected. As applications, we  give new and transparent proofs of classical results on conformal
    mappings of planar domains. In particular, we show that any conformal
    self-map of a hyperbolic planar domain with three fixed points is the
    identity. We also give a new and very simple proof of the theorem by Aumann
    and Carathéodory that
    states that the isotropy groups of a hyperbolic planar domain are either
    finite or the domain is simply-connected.
	\end{abstract}

  \blfootnote{\textup{2020} \textit{Mathematics Subject Classification}:
Primary: 30C35, 30F45, 32H50}




	\section{Introduction}

	Perhaps the most intriguing result in the theory of conformal
	mappings of planar domains proved in the last 50 years is the
	following:

	\begin{result}
		[Maskit \cite{maskit1968conformal}]\label{R:maskit} Let
		$D \subset \C$ be a hyperbolic domain and $f:D \to D$ be a
		conformal automorphism. If $f$ has three fixed points then $f$ is the
		identity.
	\end{result}
	\noindent
	This result is not explicitly stated in \cite{maskit1968conformal} but
	follows immediately from the main result there:
	\begin{result}
		\label{R: maskitmobius} Any planar domain is biholomorphic to a domain
		whose automorphisms are Möbius transformations.
	\end{result}
	There are now several proofs of Result \ref{R:maskit} available in the
	literature. Almost simultaneously in 1978/79, proofs were given by
	Leschinger \cite{Leschinger1978fixed}, Peschl and Lehtinen \cite{peschl1979conformal},
	and Minda \cite{minda1979fixed}. Minda also proves a version of Result~\ref{R:maskit}
	for Riemann surfaces and the main tool he uses is a version of Result~\ref{R:
	maskitmobius} for finite genus Riemann surfaces also proved by Maskit.
	There are other proofs of Result~\ref{R:maskit} as well; see \cite{suita1981fixed},
	\cite{fischer1987fixed}, \cite{Renggli1988} and more recently
	\cite{fridman2002fixed,krantz2005conformal}.
	The proof in the last pair of papers is particularly interesting and places the result in
	the context of the cut locus in Riemannian geometry. Further extensions of
	Result~1 to higher-dimensional hyperbolic manifolds are also studied in
	\cite{fridman2002fixed} and in \cite{fridman2007fixed}.

	Despite the glut of available proofs, the most elegant and
	aesthetically appealing argument is nevertheless still the one that
	follows immediately from Result~\ref{R: maskitmobius}. Unfortunately,
	the proof of Result~\ref{R: maskitmobius} is quite sophisticated. The
	genesis of this paper is the simple observation that if the domain $D$
	is finitely-connected then Result~\ref{R: maskitmobius} is a consequence
	of the famous generalization of the Riemann mapping theorem to finitely-connected
	domains given by Koebe.

	\begin{result}[Koebe \cite{koebe1920uniformisierung}]\label{R:koebe} Let $D$ be a finitely-connected planar
		domain. Then $D$ is biholomorphic to another domain $D'$ all of whose
		boundary components are points or circles. Furthermore, the automorphisms
		of $D'$ are Möbius transformations.
	\end{result}
	In the earlier paper \cite{koebe1908uniformisierung}, Koebe conjectured that the above result holds true for
	\textit{all} planar domains. This is known as Koebe's Kreisnormierungsproblem.
	In the seminal work of He and Schramm
	\cite{schramm93fixed}, Koebe's conjecture has been established for planar
	domains with countably many boundary components. However, their proof
	is substantially more sophisticated than Maskit's proof of Result~\ref{R:
	maskitmobius} and Koebe's conjecture is far from being established in
	its full generality.

	The main inspiration for our paper is to see if one could recover the
	elegant proof of Result~\ref{R:maskit} using Koebe's theorem instead
	of Result~\ref{R: maskitmobius}. In order to do this, we need to somehow
	reduce the proof of Result~\ref{R:maskit} to the case where $D$ is finitely-connected.
	The main result of this paper does exactly this:

	\begin{theorem}
		\label{T: kob} Let $R$ be a Kobayashi hyperbolic Riemann surface.
		Then the fundamental group of each ball under the Kobayashi
		distance is finitely generated. In particular, the balls under the
		Kobayashi distance of a hyperbolic planar domain are finitely-connected.
	\end{theorem}

	Result~\ref{R:maskit} is immediate from the above theorem: just
	consider a large Kobayashi ball that swallows up the three fixed
	points and apply Result~\ref{R:koebe}. The original argument using Result~\ref{R:
	maskitmobius} actually extends the given map to a much larger domain,
	namely $\hat{\C}$ (the extended complex plane). In contrast, the invariance of the Kobayashi
	balls under holomorphic maps allows us to \textit{restrict} the given
	map to Kobayashi balls. Theorem~\ref{T: kob} shows that these balls
	are particularly nice so that the restricted maps can be extended to
	$\hat{\C}$. This idea of restricting holomorphic maps to Kobayashi or
	Carathéodory balls in order to get global conclusions is the central
	theme of this paper. Therefore we also prove an analogous result for
	balls under the Carathéodory distance:

	\begin{theorem}
		\label{T: car} Let $D \subset \C$ be a domain that is Carathéodory
		hyperbolic. Then each relatively compact component of a ball under
		the Carathéodory distance is a finitely-connected domain.
	\end{theorem}

	Our proofs of both these theorems are not very hard. To prove
	Theorem~\ref{T: kob}, we use an argument that is a
	variant of the proof of the Van Kampen theorem. Our proof in fact
	works \textit{mutatis mutandis} for a complex manifold covered by a
	hyperbolic convex domain. As far as we are aware, the finite-connectivity of balls
	under the Kobayashi distance has not been recorded in the literature.

	As can be guessed from the statement, the proof of Theorem~\ref{T: car}
	is somewhat subtle but still not very difficult. The Carathéodory distance
	is, in general, not as well-behaved as the Kobayashi distance.
	Indeed, it has been shown recently that there are bounded planar domains
	in which the balls under the Carathéodory distance are not always connected
	(see \cite{ng2021caratheodory})! This is the reason why the statement of
	Theorem~\ref{T: car} is phrased in terms of relatively compact components.
	The main tools we use are the Separation theorem and the subharmonicity of the
	Carathéodory distance. 

	Apart from the proof of Result~\ref{R:maskit}, we give new and
	simple proofs of several classical results on conformal mappings
	using the finite-connectivity of the balls under the Kobayashi and
	Carathéodory distances. In particular, we show that the isotropy
	groups of a hyperbolic domain are either finite or the domain is
	simply-connected. (Theorem~\ref{thm:isotropy}). This is the famous
	Aumann--Carathéodory theorem \cite{aumannCaratheodary}. As far as we are
	aware, there is no simple proof of this result in the literature. A long
	and, dare we say, tedious proof can be found in \cite[Chapter~X]{burckel2021classical}. A
	more geometric and shorter proof can be found in \cite{krantz2005conformal}. The latter
	paper deals only with the bounded case and involves the construction of a
	new smooth invariant metric on the domain that extends smoothly beyond the
	boundary. In contrast, our proof uses the finite-connectivity of balls under
	the Kobayashi distance and is very short and transparent.  

  \subsection*{Organization of the paper} 
  
  In Section~\ref{S: hyperbolic}, we recall basic material about invariant
  metrics and inner distances. Section~\ref{S:canonical} contains the relevant
  facts about finitely-connected domains. The proofs of our
  main results constitute Section~\ref{S:proofs}. Finally, we present
  applications to conformal maps in Section~\ref{S:applications}

  \subsection*{Acknowledgements} We would like to thank Dr. G.P. Balakumar for
  many useful discussions and for pointing out the reference
  \cite{watt2001cartan}.

	\section{The Carathéodory and Kobayashi pseudodistances}\label{S: hyperbolic}

	In this section, we give a brief summary (without proofs) of the
	basic definitions and properties of the Carathéodory and Kobayashi
	pseudodistances. As the material is quite standard, we will not give
	precise references and instead point the reader to the books
	\cite{jp2013,kobayashi1998hyperbolic,abate2023holomorphic} for
	details.

	Our starting point is the unit disk with the Poincaré metric. Let $\rho$
	be the distance induced by the Poincaré metric. Specifically,
	\[
		\rho(z,w) = \frac{1}{2}\log\left(\frac{1 + \left|\frac{z - w}{1 -
		z\overline{w}}\right|}{1 - \left|\frac{z - w}{1 - z\overline{w}}\right|}
		\right)
	\]
	Let $M$ be a complex manifold. We now define two pseudodistances on
	$M$ using $\rho$.

	\begin{definition}[Carathéodory pseudodistance]
		The Carathéodory pseudo-distance $c_{M}$ is defined by
		\[
			c_{M}(p,p') := \sup_{f\in\hol(M,\disk)}\rho(f(p),f(p')) \text{ for
			}p, p' \in M.
		\]
	\end{definition}

	\begin{definition}[Kobayashi pseudodistance]
		Let $p, q \in M$. We choose points $p = p_{0},\ldots,p_{k} = q$ of
		$M$, points $a_{1},\ldots,a_{k},b_{1},\ldots,b_{k}$ of the unit disk,
		and holomorphic mappings $f_{1},\ldots,f_{k}$ from $\disk$ into
		$M$, such that $f_{i}(a_{i}) = p_{i-1}$ and $f_{i}(b_{i}) = p_{i}$
		for $i = 1,\ldots,k$. For each such choice of points and mappings,
		we consider the number
		\[
			\sum_{i=1}^{k} \rho(a_{i},b_{i}).
		\]
		The Kobayashi pseudodistance $d(p,q)$ is the infimum of the set of
		all such numbers obtained by varying the choice of points and mappings.
	\end{definition}
	It is not hard to show that the Carathéodory and Kobayashi pseudodistances
	on a complex manifold are indeed pseudodistances. The following
	results are easy consequences of the definitions.
	\begin{result}
		Let $M$ and $N$ be two complex manifolds, $c_{M}$ and $c_{N}$ be
		the Carathéodory pseudodistances, and $d_{M}$ and $d_{N}$ be the Kobayashi
		pseudodistances on $M$ and $N$, respectively. Then every holomorphic
		mapping $f: M \to N$ is distance decreasing in these pseudodistances.
		In particular, biholomorphisms are isometries under the Carathéodory
		and Kobayashi pseudodistances.
	\end{result}

	\begin{remark}
		The above result is the reason why the Carathéodory and Kobayashi
		pseudodistances are referred to as invariant metrics.
	\end{remark}

	\begin{result}
		For the open unit disk $\disk$, both the Kobayashi and the
		Carathéodory pseudodistances coincide with the Poincaré distance.
	\end{result}

	\begin{result}
		For any complex manifold $M$, the Kobayashi pseudodistance
		dominates the Carathéodory pseudodistance.
	\end{result}

	The next result is crucial and is the key property that makes the
	Kobayashi pseudodistance more useful than the Carathéodory pseudodistance.

	\begin{result}
		\label{P: cover} Let $M$ be a complex manifold and $(\widetilde{M},
		\pi)$ be a covering manifold. Let $p,q \in M$, and $\widetilde{p}\in
		\widetilde{M}$ such that $\pi(\widetilde{p}) = p$. Then,
		\[
			d_{M}(p,q) = \inf_{\widetilde{q} \in \pi^{-1}(q)}d_{\widetilde{M}}
			(\widetilde{p},\widetilde{q}).
		\]
	\end{result}

	Despite the non-availability of the above result for the
	Carathéodory pseudodistance, the Carathéodory pseudodistance does have
	one advantage over the Kobayashi pseudodistance:

	\begin{result}
		Let $D$ be a domain in $\C^{n}$, then $\log c_{D}$ is a
		plurisubharmonic function on $D\times D$.
	\end{result}

	\subsection{Hyperbolicity and Completeness}
	\begin{definition}[Hyperbolic Manifolds]
		Let $M$ be a complex manifold and let $d$ be the associated Kobayashi
		pseudodistance. We say that $M$ is \textit{hyperbolic} if $d$ is a
		distance. If, in addition, the metric space $(M,D)$ is complete, we
		say $M$ is \textit{complete hyperbolic}. We say $M$ is C-hyperbolic
		if the associated Carathéodory pseudodistance is a distance.
	\end{definition}

	As the Kobayashi pseudodistance always dominates the Carathéodory pseudodistance,
	the following result is immediate.
	\begin{result}
		A complex manifold is hyperbolic if the associated Carathéodory
		pseudodistance is a distance. In particular, every bounded domain in $\C^{n}$ is hyperbolic.
	\end{result}

	It is not clear just from the definitions whether the topology
	induced by the Carathéodory and Kobayashi distances coincide with
	the manifold topology assuming the manifold is hyperbolic or C-hyperbolic.
	This is in fact not true in general for the topology induced by the Carathéodory
	distance. However, we do have the following

	\begin{result}
		If $M$ a hyperbolic complex manifold then the topology induced by the
		Kobayashi distance coincides with the manifold topology. If
		$D \subset \C$ is a C-hyperbolic domain then the topology induced
		by the Carathéodory distance coincides with the Euclidean topology.
	\end{result}

	The next result gives us an abstract way of proving that a given complex
	manifold $M$ is hyperbolic.

	\begin{result}
		Let $M$ be a complex manifold and $(\widetilde{M},\pi)$ be a covering
		manifold. Then $M$ is hyperbolic iff $\widetilde{M}$ is hyperbolic.
		Furthermore, $M$ is complete hyperbolic iff $\widetilde{M}$ is
		complete hyperbolic.
	\end{result}

	\begin{remark}
		By the Uniformization theorem for Riemann surfaces, it follows that a
		Riemann surface is complete hyperbolic iff its universal cover is the
		unit disk. In this paper, we will be concerned with hyperbolic planar
		domains which, by definition, are covered by the unit disk and therefore
		are complete hyperbolic. We further note that a bounded planar domain
		has universal cover the unit disk and to prove this one does not need
		the full force of the Uniformization theorem. See \cite[Chapter~13]{zakeri2021course}.
	\end{remark}

	\subsection{Inner or Length spaces}
  
  As we have mentioned earlier, the Kobayashi distance is more useful than the
  Carathéodory distance. In any Kobayashi hyperbolic manifold, any pair of points can be
  joined by a minimal geodesic. This is not true for the Carathéodory distance.
  In this section, we state the relevant definitions and results. 

	Let $(M,d)$ be a metric space. Given a curve $\gamma:[a,b] \to M$ in
	$M$, the \textit{length} $l_{d}(\gamma)$ of $\gamma$ associated to
	the distance function $d$ is defined as follows:
	\[
		l_{d}(\gamma):=\sup \sum_{i=1}^{n} d(\gamma(t_{i-1}),\gamma(t_{i}))
		.
	\]
	where the supremum is taken over all partitions $a=t_{0}
	<t_{1}<\cdots<t_{n}=b$ of the interval $[a,b]$.

	We say a curve $\gamma$ is \textit{rectifiable} if its length
	$l_{d}(\gamma)$ is finite. A metric space $M$ is said to be \textit{finitely
	arc-wise connected} if every pair of points $x$ and $y$ in $M$ can
	be joined by a rectifiable curve. It is said to be \textit{without
	detour} if, for every $x\in M$ and $\epsilon>0$, there exists a
	$\delta>0$ such that each $y$ with $d(x,y)< \delta$ can be joined
	to $x$ by a curve $\gamma$ of length is less than $\eps$.

	\begin{definition}
		Let $(M,d)$ be a finitely arc-wise connected metric space. For
		every $x, y \in M$, define the \textit{inner distance} $d^{i}$ as
		follows;
		\[
			d^{i}(x,y):=\inf \big\{l_{d}(\gamma): \gamma \ \text{is a
			rectifiable curve joining $x$ and $y$}\big\}
		\]
	\end{definition}

	From the definition of $d^{i}$, it immediately follows that
	\[
		d(x,y)\leq d^{i}(x,y) \quad \forall x,y\in M.
	\]
	It is easy to see that $d^{i}$ is a distance function on $M$.

	\begin{result}
		Let $(M,d)$ be a finitely arc-wise connected metric space. Then
		$d$ and $d^{i}$ define the same topology on $M$ if and only if $M$
		is without detour.
	\end{result}

	\begin{result}
		Let $(M,d)$ be a finitely arc-wise connected metric space. Then
		$l_{d}(\gamma)=l_{d^i}(\gamma)$ for all curve $\gamma$ in $M$.
	\end{result}

	A metric space $(M,d)$ is said to be \textit{inner or length space} if
	$d=d^{i}$. We call such a $d$ as \textit{inner distance}. If $(M,d)$
	is a finitely arc-wise connected metric space, then $(d^{i})^{i}=d^{i}$
	and therefore $(M,d^{i})$ is an inner space. We say a metric space
	$M$ is said to be \textit{finitely compact} if every bounded infinite
	set has at least one accumulation point. A curve $\gamma$ from $x$
	and $y$ is a called \textit{minimizing geodesic} from $x$ to $y$ if $l
	_{d}(\gamma)=d(x,y)$.

	The following result shows the existence of minimizing geodesics. For a
	proof, see \cite{burago2001metric}.

	\begin{result}\label{result:min} In a finitely compact inner space, any two
	  points can be connected by a minimizing geodesic.
	\end{result}

  We note that in \cite{kobayashi1973remarks} Kobayashi himself has proved that the Kobayashi	distance on a hyperbolic complex space is an inner distance.

	\begin{result}
		\label{Kobayshi inner} Let $M$ be a hyperbolic complex manifold.
		Then the Kobayashi distance is an inner distance.
	\end{result}

	In particular, as every hyperbolic Riemann surface is an inner space and
	finitely compact, any two points in the hyperbolic
	Riemann surface can be joined by a minimizing geodesic.

	\section{Finitely-connected domains}\label{S:canonical}

	In this section, we summarize the results we need about finitely-connected
	domains. The book by Conway \cite{conway1978functions} gives a detailed
	treatment of canonical forms of such domains including a proof of Koebe's
	theorem (Result~\ref{R:koebe}). 

	\begin{definition}
		We say a domain  $D$ in $\C$ is \textit{$n$-connected}
		if $\hat{\mathbb{C}}\setminus D$ has $n+1$ components. A domain $D$
		is said to be \textit{finitely-connected} if it is $n$-connected for some
		non-negative integer $n$. 
	\end{definition}

	 At first glance, it might seem obvious that a domain whose fundamental group
	 is finitely generated is also finitely-connected. The case $n = 0$ says that
	 the complement in the extended complex plane of a simply-connected domain is connected. This is proved in a first
	 course in complex analysis using what is sometimes called the Separation
	 theorem. To deal with the case $n > 0$, we need the following more general version of
	 the Separation theorem. For a proof, see the classic book by Saks and Zygmund
	 \cite{saks1937analytic}.

	\begin{result}
		[Separation Theorem]\label{seperation theorem} Let $D$ be a
		bounded planar domain, and let $K_{1}$ and $K_{2}$ be two non-empty,
		bounded, disjoint connected components of $\hat{\C}\setminus D$. We can find
		a simple closed polygonal path $\gamma$ in $D$ such that for each $a
		\in K_{1}, b\in K_{2}$, we have
		\[
			\textsf{wind}(\gamma,a) = 1, \quad \textsf{wind}(\gamma,b) = 0.
		\]
	\end{result}

	A straightforward argument using the above result gives the following

	\begin{corollary}\label{cor:fin}
		Let $D$ be a bounded planar domain. If $\pi_{1}(D)$ is
		finitely-generated, then $D$ is finitely-connected.
	\end{corollary}




	\section{The proofs of Theorem~\ref{T: car} and Theorem~\ref{T: kob}}\label{S:proofs}

	We will first prove Theorem~\ref{T: kob}. As we remarked in the introduction,
	the proof is reminiscent of the proof of the Van Kampen theorem.

	\noindent
	\begin{proof}[Proof of Theorem~\ref{T: kob}.] Denote the Kobayashi distance
	on $R$ by $d$ and let $B$ be some open ball under $d$. By
	completeness, $\overline{B}$ is a compact subset of $R$. Let
	$p:\disk \to R$ be the universal covering map. Then there is a disk
	$D \subset \disk$ with the same radius (under the Poincaré distance)
	as $B$ such that $p(\overline{D}) = \overline{B}$. For each
	$x \in \overline{D}$, we can find another disk $D_{x}$ such that $p$
	is an isometry on $D_{x}$. We can cover $\overline{D}$ by such disks.
	By Lebesgue's covering number lemma, there exists $r > 0$ such that
	$p$ is an isometry on every open disk of disk $r$ (under the Poincaré
	distance) centered at some point of $\overline{D}$. By compactness,
	finitely many disks of radius $r/3$ centered at the points of
	$\overline{D}$ cover $\overline{D}$. Enumerate these disks as $D_{1},
	\dots, D_{n}$. The following facts are immediate.

	\begin{itemize}
		\item $D = \bigcup_{i=1}^{n} D_{i} \cap D.$

		\item If $D_{i} \cap D_{j} \cap D \neq \emptyset,$ then $D_{i} \cap
			D_{j} \cap D$ is path-connected. In fact the intersection is convex.

		\item Each $B_{i} := p(D_{i} \cap D)$ is simply-connected.
			Furthermore, $B_{i} \cap B_{j}$ is simply connected as well whenever
			it is not empty.

		\item $B = \bigcup_{i=1}^{n} B_{i}$.
	\end{itemize}

	Now let's take $\mathcal{X}$ to be a finite set containing a point from
	$B_{i}\cap B_{j}$, for each $i, j$ such that
	$B_{i}\cap B_{j}\neq \varnothing$. For each $i$ and $x, x' \in \mathcal{X}$
	such that $x,x' \in B_{i}$, denote $\sigma_{x,x'}^{i}$ to be a path
	in $B_{i}$ from $x$ to $x'$. Let $a \in \mathcal{X}$ be some point.
	Define a \textit{special loop} to be a
	loop at $a$ that is equal to a finite product of paths of the form
	$\sigma_{x,x'}^{i}$. Obviously the set of special loops
	is finite.

    Let $\gamma:[0,1] \to B$ be a loop based at $a$.
	Using compactness and by re-indexing the $B_{i}$'s, if necessary, we
	can find a partition $P=\{0=a_{0}<a_{1}<\cdots<a_{m}=1\}$ of $[0,1]$
	such that for each $k=1,2,\cdots,m$,
	\[
		\gamma([a_{k-1}, a_{k}]) \subset B_{k}, \gamma(a_{k}) \in B_{k}\cap
		B_{k+1}.
	\]
	Let $\gamma_{k}$ be the restriction of $\gamma$ to the interval $[a_{k-1}
	, a_{k}]$ reparametrized so that the domain is $[0,1]$. Let
	$z_{k} \in B_{k}\cap B_{k+1} \cap \mathcal{X}$, and $\sigma_{k}$ be a path in $B_{k}\cap
	B_{k+1}$ from $z_{k}$ to $\gamma(a_{k})$, with the understanding that
	$z_{0}=z_{m}=a$. Note that $\sigma_{0}$ and $\sigma_{m}$ are both equal
	to the constant path based at $a$, and
	\begin{align*}
		\gamma \sim \gamma_{1} * \gamma_{2} *\cdots *\gamma_{m} & \sim \sigma_{0}*\gamma_{1}*\overline{\sigma_1}*\sigma_{1}*\gamma_{2}*\overline{\sigma_2}*\cdots *\sigma_{m-1}*\gamma_{m}*\overline{\sigma_m} \\
		                                                        & \sim \Tilde{\gamma_1}*\Tilde{\gamma_2}*\cdots*\Tilde{\gamma_m}
	\end{align*}
	where, $\Tilde{\gamma_k}=\sigma_{k-1}*\gamma_{k}*\overline{\sigma_k}$.
	For each $k$, $\Tilde{\gamma_k}$ is a path in $B_{k}$ from $z_{k-1}$
	to $z_{k}$ and since $B_{k}$'s are simply connected,
	$\Tilde{\gamma_k}$ is path-homotopic to $\sigma_{z_{k-1},z_k}^{k}$. It
	follows that $\gamma$ is path-homotopic to a special loop. Hence
	$\pi_{1}(B,a)$ is finitely generated. The fact that $B$ is finitely-connected
	now follows from Corollary~\ref{cor:fin}.
\end{proof}

	\begin{remark}
		The above proof will also work in higher dimensions and the setup
		is as follows. Suppose $X$ is a complex manifold whose universal
		covering space is some hyperbolic convex domain $D$ in $\C^{n}$. Then
		the fundamental group of any Kobayashi ball in $X$ must be finitely
		generated. This follows because $X$ must be a complete hyperbolic space and every
		Kobayashi ball in $D$ is convex; see Corollary 4.8.3 in
		\cite{kobayashi1998hyperbolic} for the case that $D$ is bounded and
		Proposition 3.2 in \cite{bracci2009hyperbolicity} for the unbounded case.
	\end{remark}

	Now we come to the balls under the Carathéodory distance. It is possible
	for such a ball to be disconnected (see \cite{ng2021caratheodory}). So we will
	deal with connected components. 
	The key fact we will use is the
	subharmonicity of the Carathéodory distance. The maximum principle immediately
	yields the following

	\begin{lemma}
		\label{L:relcomp} Let $D \subset \C$ be a C-hyperbolic domain.
		Let $B$ be a relatively compact connected component of a Carathéodory ball $B(p,r)$. Then no
		component of $\hat{\C}\setminus B$ is a compact subset of $D$.
	\end{lemma}

	\begin{proof}
		Suppose $K \subset D$ is a compact component of $\hat{\C}\setminus
		B$. The separation theorem shows the existence of a simple polygonal
		loop in $B$ that encloses $K$. Let $K'$ be the union of those components
		of $\hat{\C}\setminus B$ that are enclosed by $\gamma$. Then $K'$
		is a compact set and therefore there exists a point $z \in K'$
		that maximizes $c_{D}(p,\bcdot)$ on $K'$ and this maximum is at least $r$.  
    By the maximum principle, we can find a
		point $z_{0} \in D\setminus K'$ that is enclosed by $\gamma$ with
		$c_{D}(p,z_{0}) > r$. This contradicts the definition of $K'$.
	\end{proof}

  \begin{remark}
    We have tacitly used the Jordan curve theorem for polygonal curves in the
    proof of the above lemma.
  \end{remark}

	\noindent
	\begin{proof}[Proof of Theorem~\ref{T: car}.] Let $B$ be a relatively compact
	component of the Carathéodory ball $B(p,r)$. We can find a simple
	closed polygonal path $\gamma$ in $D$ that encloses $\overline{B}$. Any
	bounded component of $\hat{\C}\setminus B$ is inside $\gamma$.
	Furthermore, no such component is a compact subset of $D$. Let $K$
	be the union of those components of $\hat{\C}\setminus D$ contained
	inside $\gamma$. Then $K$ is a compact subset of $\hat{\C}\setminus D$.
  Since $\overline{B}$ is compact, we have $d_{Euc}(\overline{B}, K) =s > 0$.
	Note that finitely many balls of radius $s/2$ cover $K$. Each bounded
	component of $\hat{\C}\setminus B$ must intersect at least one of
	these balls. Thus, there are only finitely many components of $\hat{\C}
	\setminus B$. 
  \end{proof}

	\section{Applications to Conformal Mappings}\label{S:applications}

  We now give some applications of our central results. We  confine
  ourselves to conformal mappings of planar domains. 

	Our first application is rather trivial but nevertheless
	illuminating. The result is a version of the Schwarz lemma for arbitrary
	bounded domains and is usually called the Cartan--Carathéodory
	theorem though it seems to have been first proved by Rado.

	\begin{theorem}
		Let $D \subset \C$ be a bounded domain. Let $f:D \to D$ be
		holomorphic and $f(a) = a$, for some $a \in D$. Then
		\begin{enumerate}
			\item $|f'(a)| \leq 1$.

			\item The map $f$ is an automorphism iff $|f'(a)| = 1$.

			\item If $f'(a) = 1$ then $f$ is the identity map.
		\end{enumerate}
	\end{theorem}

	\begin{proof}
		As the topology generated by the Carathéodory distance is the same
		as the Euclidean topology, we can find a Carathéodory ball
		$B(a,r)$ that is contained in an Euclidean disk whose closure is
		contained in $D$. Let $B$ be the connected component of $B(a,r)$ that
		contains $a$. Then $f(B) \subset B$. Furthermore, by 
		Lemma~\ref{L:relcomp}, $\hat{\C}\setminus B$ consists of only the single
		unbounded component as any bounded component would necessarily be a
		subset of $D$. This means $B$ is simply connected and the Schwarz
		lemma delivers the theorem.
	\end{proof}

	\begin{remark}\label{rmk:cartan}
		The idea of applying the Schwarz lemma on a simply connected open set
		that contains a fixed point and is invariant under $f$ is obviously
		not new. Indeed, the same argument works for all hyperbolic Riemann
		surfaces as the existence of a simply connected invariant open set
		is immediate from the fact the unit disk is the universal covering. The
		novelty in the above proof is in the avoidance of using the non-trivial fact
		that the universal covering of a bounded domain is the unit disk. 
	\end{remark}

  We now move on to the case when there are two fixed points. It is well-known
  that in this case the map must be an automorphism. The standard proof is using
  iteration as can be seen in \cite[Corollary~3.1.16]{abate2023holomorphic}. We
  will prove this result for bounded domains using the same idea of restriction.
  We first prove the following more general theorem that does \textit{not} seem
  amenable to a proof using iteration.

	\begin{theorem}[Watt \cite{watt2001cartan}]\label{thm:watt}
		\label{two fixed points} Let $D \subset \C$ be a bounded domain,
		$a\in D$ and $f\in \textsf{Hol}(D, D)$ with $f(a)=a$. Suppose there
		is an element $b \in D, b \neq a$ such that $d_{D}(a,b) = d_{D}(f(b),a)$,
		then $f$ must be an automorphism where $d_{D}$ is the Kobayashi distance.
	\end{theorem}

	\begin{proof}
		By Cartan's theorem, it suffices to show that $|f'(a)|=1$. By
		Result~\ref{result:min} 
		there is a minimizing geodesic in $D$
		that connects $a$ to $b$, say $\gamma:[0,1] \to D$. Then $\Tilde{\gamma}
		:=f\circ \gamma$ is a curve from $a$ to $f(b)$. We have

		\[
			l_{d_D}(\Tilde{\gamma})\geq d_{D}(f(b),a)=d_{D}(b,a)=l_{d_D}(\gamma
			)
		\]
		and the distance decreasing property gives $l_{d_D}(\Tilde{\gamma})\leq l_{d_D}(\gamma)$. It follows that $\Tilde{\gamma}$ is a
		minimal geodesic that connects $a$ and $f(b)$. If $|f'(a)| <1 $ then there exists a positive real number $R$
		such that $|f'(a)|< R < 1$. Let $D(a,r_{1})$ be a small disk around $a$
		such that $|f'(z)| < R$ for $z\in D(a,r_{1})$. We have
		\begin{equation}\label{eq:contraction}
			|f(z)-a|\leq R |z-a|, \quad \forall z\in D(a,r_{1}).
		\end{equation}

		We can find a simply-connected Kobayashi ball
		$B(a,r)$ such that
		\begin{enumerate}
			\item $B(a,r) \subset D(a,r_{1})$,
			\item $f(B(a,r)) \subset B(a,r')$ where $r' < r$.
		\end{enumerate}
		This is immediate from
		\eqref{eq:contraction} and the fact that the universal covering map of $D$ is a
		local isometry of the Poincaré distance on $\D$ and the Kobayashi
		distance on $D$.

		Now by the continuity of $\gamma$ and $d_{D}(a, \Tilde{\gamma}(t))$, there
		is a $t_{0} \in [0,1]$ such that
		\[
			d_{D}(a,\gamma(t_{0}))=r \quad \text{and}\quad d_D(a, f(\gamma(t_{0}
			))=d_{D}(a, \Tilde{\gamma}(t_{0}))\leq r'. 
		\]
		We have the following chain of inequalities:
		\begin{align*}
			d_{D}(a, f(b))&\leq d_D(a, f(\gamma(t_{0})))+d_D(f(\gamma(t_{0})), f(b
			))\\
			&\leq r' + d_D(\gamma(t_{0}),b) \leq r' + d_D(a,
			b)-r <d_D(a,b).
		\end{align*}
		This is a contradiction. Thus $|f'(a)|=1$ and $f$ must an automorphism.
	\end{proof}

	\begin{remark}
	  The existence of a minimizing geodesic is \textbf{not} essential for
	  the proof though it
	  simplifies the argument. What is really used is the fact that the Kobayashi
	  distance is an inner distance. It is easy to modify the proof to rely
	  only on this latter fact. 
	\end{remark}

	\begin{corollary}
		Let $f$ be a holomorphic self-map of $D$ with two fixed points.
		Then $f$ must be an automorphism. 
	\end{corollary}

 We now come to the proof of Maskit's theorem.

 \begin{proof}[Proof of Result~\ref{R:maskit}] Let $B \subset D$ be a
 Kobayashi ball such that three fixed points of $f$ are in $B$. By completeness,
 $B$ is bounded and by the distance decreasing property of $f$, $f(B) \subset
 B$. The previous corollary shows that $f$ is an automorphism. As $B$ is
 finitely-connected, by Koebe's theorem $B$ is biholomorphic to a domain whose
 boundary components are points or circles and $f$ induces an automorphism of
 this domain whence it  must be a Möbius transformation. As any Möbius
 transformation with three fixed points is the identity, $f$ must be the
 identity as well.
 \end{proof}

	The next result is about the isotropy group of a hyperbolic planar
	domain. The setup is as follows: for fixed $p\in D$, denote $\textsf{
	Iso}_{p}(D)$ to be the subgroup consisting of those automorphisms of
	$D$ that fixes $p$. 
	\begin{theorem}\label{thm:isotropy}
		Let $D$ be a hyperbolic planar domain that is \textit{not} simply connected
		and let $p\in D$. Then the isotropy group $\textsf{Iso}_{p}(D)$ is
		finite cyclic. 
	\end{theorem}
	\begin{proof}
		Let $B$ be a Kobayashi ball around $p$ with radius $r$ chosen so large that
		$B$ is an $n$-connected domain with $n \geq 1$. By Koebe's theorem, it
		follows that $B$ is biholomorphic to a domain $G$ whose boundary components
		are points or circles. Any element of $\textsf{Iso}_{p}(D)$ induces
		an automorphism of $G$. The induced automorphism is a M{\"o}bius
		transformation that fixes a point and permutes the boundary components of
		$G$ of which there are at least $2$. In fact we may assume that the
		`outer' boundary of $G$ is the unit circle. It follows that there are only
		finitely many elements in $\textsf{Iso}_{p}(D)$. 
    
    According to
		Cartan's theorem (see Remark~\ref{rmk:cartan}), the mapping $f \mapsto |f'(p)|$ from
		$\textsf{Iso}_{p}(D)$ to the unit circle is an injective group
		homomorphism. Hence $\textsf{Iso}_{p}(D)$ is cyclic since it is finite.
	\end{proof}

	\begin{remark}
		Any finitely-connected domain (without punctures) is biholomorphic
		to a circular slit domain; see \cite[Chapter~6.5]{ahlfors1979complex}. This
		fact is simpler to prove than Koebe's theorem. It is
		straightforward to modify the above proof as well as the proof of
		Result~\ref{R:maskit} to rely on this fact instead of Koebe's theorem.
	\end{remark}

	\bibliographystyle{amsalpha}
	\bibliography{conformal}
\end{document}